\documentclass[a4paper,10pt]{amsart}

\usepackage[english]{babel}
\usepackage[utf8]{inputenc}

\usepackage{amssymb}
\usepackage{amsmath}
\usepackage{amsthm}
\usepackage{bbm}

\usepackage{enumerate}
\usepackage{tikz}
\usepackage{marginnote}

\newcommand{\bbC}{\mathbb{C}}

\newcommand{\bbN}{\mathbb{N}}

\newcommand{\bbR}{\mathbb{R}}

\newcommand{\calK}{\mathcal{K}}
\newcommand{\calL}{\mathcal{L}}

\newcommand{\calU}{\mathcal{U}}

\DeclareMathOperator{\one}{\mathbbm{1}}

\newcommand{\argument}{\mathord{\,\cdot\,}}
\newcommand{\dx}{\;\mathrm{d}}

\DeclareMathOperator{\fix}{fix}
\newcommand{\norm}[1]{\left\lVert #1 \right\rVert}
\newcommand{\modulus}[1]{\left\lvert #1 \right\rvert}

\newcommand{\spec}{\sigma}
\newcommand{\Res}{\mathcal{R}}

\newcommand{\sprEss}{r_{\operatorname{ess}}}
\newcommand{\specEss}{\sigma_{\operatorname{ess}}}

\theoremstyle{definition}
\newtheorem{definition}{Definition}[section]
\newtheorem{remark}[definition]{Remark}

\newtheorem*{remark_no_number}{Remark}
\newtheorem*{remarks_no_number}{Remarks}

\theoremstyle{plain}
\newtheorem{proposition}[definition]{Proposition}
\newtheorem{lemma}[definition]{Lemma}
\newtheorem{theorem}[definition]{Theorem}
\newtheorem{corollary}[definition]{Corollary}

\numberwithin{equation}{section}

\begin{document}

\title[Spectral gaps for hyperbounded operators]{Spectral gaps for hyperbounded operators}
\author{Jochen Gl\"uck}
\address{Jochen Gl\"uck, Universit\"at Passau, Fakultät für Informatik und Mathematik, Innstra{\ss}e 33, 94032 Passau, Germany}
\email{jochen.glueck@uni-passau.de}
\date{\today}
\begin{abstract}
	We consider a positive and power-bounded linear operator $T$ on $L^p$ over a finite measure space and prove that, if $TL^p \subseteq L^q$ for some $q > p$, then the essential spectral radius of $T$ is strictly smaller than $1$. As a special case, we obtain a recent result of Miclo who proved this assertion for self-adjoint ergodic Markov operators in the case $p=2$ and thereby solved a long-open problem of Simon and H{\o}egh-Krohn. 
	
	Our methods draw a connection between spectral theory and the geometry of Banach spaces: they rely on a result going back to Groh that encodes spectral gap properties via ultrapowers, and on the fact that an infinite dimensional $L^p$-space cannot by isomorphic to an $L^q$-space for $q \not= p$.
	
	We also prove a number of variations of our main result: (i) it follows from theorems of Lotz and Mart\'{i}nez that the condition $TL^p \subseteq L^q$ can be replaced with the weaker assumption that $T$ maps the positive part of the $L^p$-unit ball into a uniformly $p$-integrable set; (ii) while it is known that the positivity assumption on $T$ cannot in general be omitted, we show that we can replace it with the assumption that $T$ is contractive both on $L^p$ and on $L^q$; (iii) we prove a version of the theorem which allows us, under appropriate circumstances, to also consider non-finite measures spaces; (iv) our result also has a uniform version: there exists an upper bound $c \in [0,1)$ for the essential spectral radius of $T$, where $c$ depends on certain quantitative properties of $T$, $L^p$ and $L^q$.
\end{abstract}

\maketitle

\section{Introduction}

The \emph{essential spectral radius} $\sprEss(T)$ of a bounded linear operator $T$ on a complex Banach space is defined to be the smallest number $r$ in $[0,\infty)$ such that every complex number of modulus $> r$ is either in the resolvent set of $T$ or a pole of the resolvent of $T$ and an eigenvalue of $T$ of finite algebraic multiplicity. The aim of this paper is to prove the following result, as well as several variations of it.

\begin{theorem} \label{thm:main-result-introduction}
	Let $1 < p < q < \infty$ and let $(\Omega,\mu)$ be a finite measure space. If a positive and power-bounded linear operator $T: L^p(\Omega,\mu) \to L^p(\Omega,\mu)$ satisfies $TL^p(\Omega,\mu) \subseteq L^q(\Omega,\mu)$, then $\sprEss(T) < 1$.
\end{theorem}

Here, \emph{positive} means that $Tf \ge 0$ almost everywhere whenever $f \in L^p(\Omega,\mu)$ satisfies $f \ge 0$ almost everywhere, and \emph{power-boundedness} means that the \emph{power bound} $\sup_{n \in \bbN_0} \norm{T^n}$ is finite. The above theorem is actually a special case of two more general results below (Corollary~\ref{cor:main-result-without-power-boundedness} and Theorem~\ref{thm:main-result-positive}).

The relation between Theorem~\ref{thm:main-result-introduction} and the title of our paper is as follows: in the situation of the theorem the property $TL^p(\Omega,\mu) \subseteq L^q(\Omega,\mu)$ is usually referred to as \emph{hyperboundedness} of $T$. Moreover, we say that a bounded linear operator $T$ of spectral radius $1$ on a complex Banach space has a \emph{spectral gap} if the number $1$ is a spectral value and a pole of the resolvent of $T$. Thus, the conclusion of Theorem~\ref{thm:main-result-introduction} implies that either the spectral radius of $T$ is strictly smaller than $1$ or that $T$ has a spectral gap.

\subsection*{Historical remarks and related literature}

A special case of Theorem~\ref{thm:main-result-introduction} was recently proved by Miclo \cite[Theorem~1]{Miclo2015}; he assumed in addition that $p = 2$, that $T$ is self-adjoint and that the fixed space $\ker(1-T)$ is spanned by the constant function with value $1$. Miclo's result was a breakthrough, solving a long open problem of Simon and H{\o}egh-Krohn \cite{Simon1972} which arose in the 1970s in the context of mathematical physics. Of all the related articles that engaged this problem during the last decades, let us mention the works \cite{Aida1998, Hino2000, Wu2001, Hino2002, Gong2006} which impose a strengthened positivity assumption on the operator and the paper \cite{Wang2004} which assumes explicit numerical bounds for the operator norm of $T$ from $L^p$ to $L^q$. Moreover, we note that the technique used by Miclo was further developed in the recent article \cite{Wang2014}.

Quite interestingly, it turns out that Theorem~\ref{thm:main-result-introduction} can be obtained as a consequence of an old, but hardly known result of Lotz \cite[Corollary~7]{Lotz1986} about positive operators on Banach lattices (and a slightly later result of Mart\'{i}nez \cite[Theorem~2.6]{Martinez1993} even implies a generalisation of Theorem~\ref{thm:main-result-introduction}). We discuss this, along with related results about \emph{uniform $p$-integrability}, in Section~\ref{section:uniform-p-integrability}. In the rest of the paper we then provide a more self-contained proof of Theorem~\ref{thm:main-result-introduction}, as well as several variations of this theorem, based on geometric properties of $L^p$-spaces.

\begin{remark_no_number}
	We should mention here that, when comparing our results to those in \cite{Miclo2015}, one has to be a bit careful about the notion \emph{spectral gap}. In \cite{Miclo2015} a slightly stronger definition is used: there an operator $T$ is said to have a spectral gap if, in addition to the properties occurring in our definition above, the eigenspace $\ker(1-T)$ is spanned by the constant function $\one$. However, it is important to note that the property $\ker(1-T) = \bbC \cdot \one$ enters the main result of \cite{Miclo2015} both as an assumption and -- encoded in the notion \emph{spectral gap} -- as a conclusion.
	
	We do not use an assumption of the form $\ker(1-T) = \bbC \cdot \one$ which is why, in the context of the present paper, our slightly less restrictive definition of a \emph{spectral gap} appears to be more appropriate.
\end{remark_no_number}

\subsection*{Techniques, further results and organisation of the paper}

The proof in \cite{Miclo2015} relies on Cheeger inequalities on finite graphs and an approximation procedure. Our approach is very different. Our central ingredient is an ultrapower technique which goes originally back to Groh \cite[Proposition~3.2]{Groh1984} and was later refined by Caselles \cite[Proposition~3.3]{Caselles1987}. It allows us, under mild conditions, to prove that a spectral value $\lambda$ of an operator $T$ is a pole of the resolvent $\Res(\argument,T)$ (with finite-rank residuum) if $\ker(\lambda-T^\calU)$ is finite dimensional for a certain ``lifted'' operator $T^\calU$. We explain this in detail in Section~\ref{section:spectral-gaps-via-ultra-powers}.

The same kind of ultrapower argument was used by Lotz \cite[Corollary~7]{Lotz1986} and Mart\'{i}nez \cite[Theorem~2.6]{Martinez1993} to derive sufficient criteria for a positive operator $T$ on an $L^p$-space to satisfy $\sprEss(T) < 1$; we use these criteria, which rely on rather involved Banach lattice techniques, as black boxes in Section~\ref{section:uniform-p-integrability} to prove (a generalisation of) Theorem~\ref{thm:main-result-introduction}. Then, after giving a detailed account of the relevant ultrapower technology in Section~\ref{section:spectral-gaps-via-ultra-powers}, we provide another, more self-contained proof of Theorem~\ref{thm:main-result-introduction} in Section~\ref{section:spectral-gaps-for-hyperbounded-operators}; this proof is based on the observation that an operator which maps $L^p$ to $L^q$ must have a finite dimensional fixed space as a consequence of the geometric properties of $L^p$- and $L^q$-spaces. Our argument thus draws a connection from the geometry of Banach spaces to the spectral theory of linear operators. A similar connection, though with a different technical flavour, has recently been exploited in \cite{Glueck2016b}.

Our approach also allows us to vary Theorem~\ref{thm:main-result-introduction} in several respects. First of all we prove a somewhat more general result (Theorem~\ref{thm:main-result-positive}) which also allows us to consider $L^p$-spaces over non-finite measure spaces in certain circumstances (Corollary~\ref{cor:main-result-infinite-measure-space}). Further, it was shown by Wang in \cite[Theorem~1.4]{Wang2004} that the positivity assumption cannot be removed from Theorem~\ref{thm:main-result-introduction}. However, we prove that the positivity assumption can be replaced with an additional contractivity assumption (Theorem~\ref{thm:main-result-contractive}) or with a hypercontractivity assumption (Corollary~\ref{cor:hypercontractive-operators}). Moreover, such contractivity assumptions can be slightly weakened if $p = 2$ or $q = 2$ (Remark~\ref{rem:main-result-hilbert}).

In Section~\ref{section:uniform-estimates} we demonstrate how uniform estimates for the essential spectral radius of hyperbounded operators can be derived from our results by means of an ultraproduct argument. Section~\ref{section:applications} gives an application of our results to the long-term behaviour of positive operator semigroups. In the concluding Section~\ref{section:concluding-remarks} we discuss the case $p = 1$ as well as the case $q = \infty$, and we make a brief remark about the situation on $\ell^p$-spaces. In the appendix we recall several geometric properties of $L^p$-spaces which are needed in our proofs.

\subsection*{Preliminaries}

We assume the reader to be familiar with real and complex Banach lattices; standard references for this theory are, for instance, the monographs \cite{Schaefer1974, Zaanen1983, Meyer-Nieberg1991}. Here we only recall the basic terminology that a linear operator $T$ on a Banach lattice $E$ is called \emph{positive} if $Tf \ge 0$ for each $0 \le f \in E$.

The space of bounded linear operators on a Banach space $E$ is denoted by $\calL(E)$; we endow this space with the operator norm topology throughout. The reader is also assumed to be familiar with standard spectral theory for linear operators; for a detailed treatment we refer, for instance, to the spectral theory chapters in the monographs \cite{Taylor1958, Kato1995, Yosida1995}. If $T$ is a bounded linear operator on a complex Banach space $E$ we denote the \emph{spectrum} of $T$ by $\spec(T)$; if $\lambda \in \bbC$ is not a spectral value of $T$, then we denote the \emph{resolvent} of $T$ at $\lambda$ by $\Res(\lambda,T) := (\lambda - T)^{-1}$. The \emph{essential spectrum} of $T$ is defined as the set $\specEss(T) := \{\lambda \in \bbC: \, \lambda - T \text{ is not Fredholm}\}$. The essential spectrum is compact and, if $E$ is infinite dimensional, non-empty since $\specEss(T)$ coincides with the spectrum of the equivalence class of $T$ in the \emph{Calkin algebra} $\calL(E) / \calK(E)$ (where $\calK(E) \subseteq \calL(E)$ denotes the algebraic ideal of all compact linear operators on $E$). The essential spectral radius $\sprEss(T)$ as defined at the very beginning of the paper coincides with the supremum of the modulus of all numbers in $\specEss(T)$. For any $r > \sprEss(T)$ only finitely many spectral values of $T$ can have modulus at least $r$.

The \emph{fixed space} of an operator $T \in \calL(E)$ on a Banach space $E$ is defined as $\fix T := \ker(1-T)$. We call a bounded linear operator $T$ \emph{contractive} if its operator norm satisfies $\norm{T} \le 1$.

We make extensive use of ultrapower arguments throughout the article; some important facts about the construction of ultrapowers are briefly recalled in Subsection~\ref{subsection:ultra-powers}. For a detailed treatment of ultrapowers and ultraproducts of Banach spaces we refer the reader to the survey article \cite{Heinrich1980}. 

Throughout the paper, all occurring Banach spaces and Banach lattices are assumed to be defined over the complex scalar field. All measure spaces in the paper are allowed to be non-$\sigma$-finite unless otherwise stated.

\section{Uniform $p$-integrability} \label{section:uniform-p-integrability}

In this section we demonstrate that Theorem~\ref{thm:main-result-introduction} follows easily from an old but apparently very little known result of Lotz \cite[Corollary~7]{Lotz1986}, and that it can -- by a more general result of Mart\'{i}nez \cite[Theorem~2.6]{Martinez1993} -- even be generalised to operators which are not power-bounded but merely have spectral radius at most $1$.

Let $p \in [1,\infty)$ and let $(\Omega,\mu)$ be a finite measure space. Recall that a subset $S \subseteq L^p(\Omega,\mu)$ is called \emph{uniformly $p$-integrable} if
\begin{align*}
	\sup_{f \in S} \int_{\{\modulus{f} \ge r\}} \modulus{f}^p \dx\mu \to 0
\end{align*}
as the real number $r$ tends to $\infty$. If $q > p$, then the unit ball of $L^q(\Omega,\mu)$ is uniformly $p$-integrable; this follows easily from H\"older's inequality.

Wu \cite{Wu2000} considered positive operators $T$ on $L^p$ which map the unit ball into a uniformly $p$-integrable set, and showed that this has considerable consequences for the spectral properties of $T$. Uniform $p$-integrability is also very instrumental in order to derive Theorem~\ref{thm:main-result-introduction} from the following result that was shown by Lotz \cite[Corollary~7 on p.\,152]{Lotz1986} for Ces\`{a}ro bounded operators and, a few years later, by Mart\'{i}nez \cite[Theorem~2.6]{Martinez1993} for the general case.

\begin{theorem} \label{thm:lotz-martinez}
	Let $p \in [1,\infty)$ and let $(\Omega,\mu)$ be an arbitrary measure space. Let $T$ be a positive linear operator on $L^p := L^p(\Omega,\mu)$ with spectral radius at most $1$, and assume that there exists an integer $m \ge 1$, a real number $\eta \in [0,1)$ and a function $0 \le u \in L^p$ such that
	\begin{align}
		\label{eq:dual-doeblin-positive}
		T^m B_+ \subseteq \eta B_+ + [0,u];
	\end{align}
	here, $B_+ = \{f \in L^p: \, \norm{f}_{L^p} \le 1 \text{ and } f \ge 0\}$ denotes the positive part of the closed unit ball in $L^p$ and $[0,u] = \{f \in L^p: \, 0 \le f \le u\}$ denotes the order interval between $0$ and $u$.
	
	Then the essential spectral radius of $T$ satisfies $\sprEss(T) < 1$.
\end{theorem}

\begin{remark}
	Rather than~\eqref{eq:dual-doeblin-positive}, Mart\'{i}nez actually assumed the slightly different condition
	\begin{align}
		\label{eq:dual-doeblin-non-positive}
		T^mB \subseteq \eta B + B_u,
	\end{align}
	where $B$ denotes the closed unit ball in $L^p$ and where $B_u$ denotes the set $\{f \in L^p: \, \modulus{f} \le u\}$. However, both conditions are equivalent in the following sense:
	\begin{enumerate}[\upshape (a)]
		\item Condition~\eqref{eq:dual-doeblin-non-positive} implies~\eqref{eq:dual-doeblin-positive} (for the same $m$, $\eta$ and $u$). This follows from the so-called \emph{Riesz decomposition property} $[0,u+v] = [0,u] + [0,v]$ which is true for all functions $u,v \ge 0$.
		
		\item Condition~\eqref{eq:dual-doeblin-positive} implies~\eqref{eq:dual-doeblin-non-positive} (with different $m$, $\eta$ and $u$). To see this, we iterate \eqref{eq:dual-doeblin-positive} sufficiently often to obtain $T^{nm}B_+ \subseteq \frac{1}{8}B_+ + [0, \tilde u]$. Hence, $T^{nm}B \subseteq \frac{1}{2}B + B_{4\tilde u}$ by the triangle inequality for the complex modulus.
	\end{enumerate}
\end{remark}

Theorem~\ref{thm:lotz-martinez} is a powerful result which has the following two corollaries; the second of them contains Theorem~\ref{thm:main-result-introduction} as a special case.

\begin{corollary} \label{cor:uniformly-p-integrable}
	Let $p \in [1,\infty)$ and let $(\Omega,\mu)$ be a finite measure space. Let $T$ be a positive linear operator on $L^p := L^p(\Omega,\mu)$ with spectral radius at most $1$ and assume that, for some $m \in \bbN$, the power $T^m$ maps $B_+$ (the positive part of the unit ball) to a uniformly $p$-integrable set. Then $\sprEss(T) < 1$.
\end{corollary}
\begin{proof}
	The uniform $p$-integrability of $T^mB_+$ implies that there exists a real number $r \ge 0$ such that $\int_{f \ge r} f^p \dx \mu \le \frac{1}{2}$ for each $f \in T^mB_+$. In particular,
	\begin{align*}
		\norm{f - f \land r\one}_{L^p} \le \left(\int_{f \ge r} f^p \dx \mu \right)^{1/p} \le \left(\frac{1}{2}\right)^{1/p}
	\end{align*}
	for each $f \in T^mB_+$, so it follows that $T^m B_+ \subseteq \left(\frac{1}{2}\right)^{1/p} B_+ + [0, r\one]$. The assertion now follows from Theorem~\ref{thm:lotz-martinez}.
\end{proof}

We note that Corollary~\ref{cor:uniformly-p-integrable} gives a positive answer to a question raised in \cite[Remark~3.6]{ElMachkouri2019}

\begin{corollary} \label{cor:main-result-without-power-boundedness}
	Let $1 < p < q < \infty$ and let $(\Omega,\mu)$ be a finite measure space. Let $T$ be a positive linear operator on $L^p := L^p(\Omega,\mu)$ with spectral radius at most $1$ and assume that, for some $m \in \bbN$, we have $T^mL^p \subseteq L^q := L^q(\Omega,\mu)$. Then $\sprEss(T) < 1$.
\end{corollary}
\begin{proof}
	It follows from the closed graph theorem that $T^m$ is continuous from $L^p$ to $L^q$, so $T^m$ maps the $L^p$-unit ball into a multiple of the $L^q$-unit ball, which is uniformly $p$-integrable. The assertion thus follows from Corollary~\ref{cor:uniformly-p-integrable}.
\end{proof}

Note that Corollary~\ref{cor:main-result-without-power-boundedness} remains true for the cases $p = 1$ or $q = \infty$ (with the same proof). The reason why we excluded these cases from the statement of the corollary is that much more is actually true if $p=1$ or $q=\infty$; we explain this in Subsection~\ref{subsection:end-points-of-the-l_p-scale} at the end of the paper.

Corollary~\ref{cor:main-result-without-power-boundedness} shows that we can replace the power-boundedness assumption in Theorem~\ref{thm:main-result-introduction} with the weaker assumption that the spectral radius of $T$ is at most $1$. The decision whether one imposes the condition that $L^p$ is mapped into $L^q$ on $T$ itself (as in Theorem~\ref{thm:main-result-introduction}) or on some power of $T$ (as in Corollary~\ref{cor:main-result-without-power-boundedness}) is mainly a matter of taste, since we have $\sprEss(T) = \sprEss(T^m)^{1/m}$ for each $m \in \bbN$.

Two drawbacks of the approach presented in the section are that (i) the theorems of Lotz and Mart\'{i}nez use quite technical machinery from the theory of Banach lattices in their proofs and that (ii) this approach does not permit generalisations to non-positive operators. In the rest of this article we present an approach that is, similar as the theorems of Lotz and Mart\'{i}nez, based on an ultrapower argument that originally goes back to Groh \cite[Proposition~3.2]{Groh1984}, but we replace condition~\eqref{eq:dual-doeblin-positive} with geometric properties of $L^p$-spaces. We thus present a more self-contained proof of Theorem~\ref{thm:main-result-introduction} (and a generalization of it in Theorem~\ref{thm:main-result-positive}) as well as a version of the theorem where positivity of the operator is replaced with a contractivity assumption.

\section{Spectral gaps via ultrapowers} \label{section:spectral-gaps-via-ultra-powers}

In this section we show that a spectral value $\lambda$ of a linear operator $T$ is, under mild assumptions, a pole of the resolvent if the corresponding eigenspace of a ``lifted'' operator $T^\calU$ has finite dimension. For power-bounded $T$ and $\modulus{\lambda} = 1$ this was first proved by Groh \cite[Proposition~3.2]{Groh1984}; later on, Caselles \cite[Proposition~3.3]{Caselles1987} used the theory of Fredholm operators to show that the same result remains true under more general assumptions. In Corollary~\ref{cor:eigenspace-on-ultra-power} we give a slightly more general version of Caselles' result; the main thrust of the argument is still the same, though.

\subsection{A brief reminder of ultrapowers} \label{subsection:ultra-powers}

Let us briefly recall the concept of an ultrapower of a Banach space $E$. Fix a free ultrafilter $\calU$ on $\bbN$, endow the $E$-valued sequence space $\ell^\infty(E)$ with its canonical norm $\norm{z} := \sup_{n \in \bbN} \norm{z_n}_E$ for $z = (z_n)_{n \in\bbN} \in \ell^\infty(E)$, and define
\begin{align*}
	c_{0,\calU}(E) := \{z \in \ell^\infty(E): \; \lim_{n \to \calU} \norm{z_n}_E = 0\}.
\end{align*}
Then $\ell^\infty(E)$ is a Banach space and $c_{0,\calU}(E)$ is a vector subspace of it. The quotient space
\begin{align*}
	E^\calU := \ell^\infty(E)/c_{0,\calU}(E)
\end{align*}
is called the \emph{ultrapower} of $E$ with respect to the ultrafilter $\calU$. For each $z = (z_n)_{n \in \bbN} \in \ell^\infty(E)$ we use the notation $z^\calU$ for the equivalence class of $z$ in $E^\calU$. Moreover, for $x \in E$ we use the notation $x^\calU$ for the equivalence class of the constant sequence $(x)_{n \in \bbN} \in \ell^\infty(E)$ in $E^\calU$. The mapping $E \ni x \mapsto x^\calU \in E^\calU$ is isometric, and via this mapping we may consider $E$ as a closed subspace of $E^\calU$. For every $z \in \ell^\infty(E)$ we can compute the norm of $z^\calU$ in $E^\calU$ by means of the formula $\norm{z^\calU} = \lim_{n \to \calU} \norm{z_n}_E$.

If $E$ is a Banach lattice, then so is $\ell^\infty(E)$, and then the space $c_{0,\calU}(E)$ is a closed ideal in $\ell^\infty(E)$. Thus, the ultrapower $E^\calU$ is a Banach lattice, too, and the embedding $E \ni x \mapsto x^\calU \in E^\calU$ is an isometric lattice homomorphism in this case. The formula $\norm{z^\calU} = \lim_{n \to \calU} \norm{z_n}_E$ for $z \in \ell^\infty(E)$ implies that, if $E$ is an $L^p$-space over some measure space for $p \in [1,\infty)$, then $E^\calU$ is an abstract $L^p$-space and thus isometrically lattice isomorphic to a concrete $L^p$-space by the representation theorem for abstract $L^p$-spaces \cite[Theorem~2.7.1]{Meyer-Nieberg1991}.

Let $E,F$ be Banach spaces. Every bounded linear operator $T: E \to F$ can be canonically extended to a bounded linear operator $T^\calU: E^\calU \to F^\calU$ which is given by $T^\calU z^\calU = (Tz_n)^\calU$ for each $z = (z_n) \in \ell^\infty(E)$. If $E = F$, then the mapping $T \mapsto T^\calU$ is an isometric and unital Banach algebra homomorphism from $\calL(E)$ into $\calL(E^\calU)$. If $E$ and $F$ are Banach lattices, then $T$ is positive if and only if $T^\calU$ is positive; likewise, $T$ is a lattice homomorphism if and only if $T^\calU$ is a lattice homomorphism.

Ultrapowers are an important tool in operator theory. One of their very useful properties is that the lifting $T \mapsto T^\calU$ improves the behaviour of certain parts of the spectrum of $T$ without changing the spectrum as a whole. For instance,  every approximate eigenvalue of $T$ is an eigenvalue of $T^\calU$; see e.g.\ \cite[Section~V.1]{Schaefer1974} and \cite[Theorem~4.1.6]{Meyer-Nieberg1991} for details. In the next subsection we discuss another useful fact, namely that the dimension of the eigenspaces of $T^\calU$ contains information about certain spectral properties of $T$ (Corollary~\ref{cor:eigenspace-on-ultra-power}).

\subsection{Eigenspaces on ultrapowers} \label{subsection:eigenspaces-on-ultra-powers}

Corollary~\ref{cor:eigenspace-on-ultra-power} of the following theorem is the key ingredient for the proofs of our main results in Section~\ref{section:spectral-gaps-for-hyperbounded-operators}.

\begin{theorem} \label{thm:kernel-on-ultra-power}
	Let $E,F$ be Banach spaces, let $T: E \to F$ be a bounded linear operator and let $\calU$ be a free ultrafilter on $\bbN$. If the kernel of $T^\calU$ is finite dimensional, then the range of $T$ is closed.
\end{theorem}

As mentioned before, the following proof is based on an idea taken from \cite[Proposition~3.2]{Groh1984} (see also \cite[Lemma~C-III-3.10]{Arendt1986}). Moreover, the author is indebted to Manfred Wolff for several suggestions which eventually lead to the above very general form of the theorem.

\begin{proof}[Proof of Theorem~\ref{thm:kernel-on-ultra-power}]
	Assume that the range of $T$ is not closed in $F$ and let $\tilde T: E/\ker T \to F$ denote the canonical operator induced by $T$. The range of $\tilde T$ is not closed (as it coincides with the range of $T$), so $\tilde T$ is not bounded below. Hence, there exists a sequence $(\tilde x_n)_{n \in \bbN} \subseteq E/\ker T$ of normalised vectors such that $\tilde T \tilde x_n \to 0$ in $F$ as $n \to \infty$.
	
	Every vector $\tilde x_n$ has a representative $x_n$ in $E$ whose norm is not larger than $2$. Note that the distance of $x_n$ to $\ker T$ equals $1$ and that $Tx_n = \tilde T \tilde x_n \to 0$ as $n \to \infty$. Moreover, the sequence $(x_n)_{n \in \bbN}$ does not have any convergent subsequence, for if $x \in E$ was the limit of such a subsequence then $x \in \ker T$, which contradicts the fact that each $x_n$ has distance $1$ to $\ker T$.
	
	Hence, the set $\{x_n: \, n \in \bbN\}$ is not relatively compact in $E$, so there exists a number $\varepsilon > 0$ such that $\{x_n: \, n \in \bbN\}$ cannot be covered by finitely many open balls of radius $\varepsilon$. Therefore we can find a subsequence $(y_n)_{n \in \bbN}$ of $(x_n)_{n \in \bbN}$ such that $\norm{y_n - y_m} \ge \varepsilon$ for all distinct $m,n \in \bbN$. For each $k \in \bbN$ we define
	\begin{align*}
		y^{(k)} := \big((y_{n+k})_{n \in \bbN}\big)^\calU \in E^\calU.
	\end{align*}
	Then each $y^{(k)}$ is an element of $\ker(T^\calU)$ and its norm is between $1$ and $2$. However, for $j \not= k$ we obtain
	\begin{align*}
		\norm{y^{(k)} - y^{(j)}} = \lim_{n \to \calU} \norm{y_{n+k} - y_{n+j}} \ge \varepsilon,
	\end{align*}
	so the sequence $(y^{(k)})_{k \in \bbN}$ does not have a convergent subsequence. Hence, $\ker(T^\calU)$ is infinite dimensional.
\end{proof}

\begin{corollary} \label{cor:eigenspace-on-ultra-power}
	Let $T$ be a bounded linear operator on a Banach space $E$ and let $\lambda \in \bbC$ be an element of the topological boundary of $\spec(T)$. Assume that, for a free ultrafilter $\calU$ on $\bbN$, the eigenspace $\ker(\lambda - T^\calU)$ is finite dimensional.
	
	Then $\lambda$ is a pole of the resolvent $\Res(\argument,T)$ and the algebraic multiplicity of $\lambda$ as an eigenvalue of $T$ is finite.
\end{corollary}

It is not difficult to see that the converse implication is also true in Corollary~\ref{cor:eigenspace-on-ultra-power}; since we do not need this observation in what follows, we do not discuss this in detail, though.

\begin{proof}[Proof of Corollary~\ref{cor:eigenspace-on-ultra-power}]
	By Theorem~\ref{thm:kernel-on-ultra-power} the operator $\lambda - T$ has closed range. The kernel of $\lambda-T$ is finite-dimensional as it can be identified with a subspace of $\ker(\lambda - T^\calU)$, so $\lambda - T$ is an \emph{upper semi-Fredholm operator}. 
	
	It is well-known that the set of semi-Fredholm operators is open in $\calL(E)$ and that the Fredholm index is constant on connected components of this set, see for instance \cite[Theorem~IV.5.17 on p.\,235]{Kato1995}. Since $\lambda$ is a boundary point of $\spec(T)$ we can find a sequence of bijective operators in $\calL(E)$ which approximates $\lambda - T$ in operator norm. Hence, $\lambda - T$ is actually a Fredholm operator with Fredholm index zero. The assertion thus follows from analytic Fredholm theory, see for instance \cite[Corollary~XI.8.4 on p.\,203]{Gohberg1990}.
\end{proof}

\section{Spectral gaps for hyperbounded operators} \label{section:spectral-gaps-for-hyperbounded-operators}

\subsection{Positive operators}

Theorem~\ref{thm:main-result-introduction} in the introduction is a special case of the following more general result.

\begin{theorem} \label{thm:main-result-positive}
	Let $p,q \in (1,\infty)$ be two distinct numbers and let $(\Omega_1,\mu_1)$ and $(\Omega_2,\mu_2)$ be two arbitrary measure spaces. Moreover, let $j: L^q := L^q(\Omega_2,\mu_2) \to L^p := L^p(\Omega_1,\mu_1)$ be a lattice homomorphism (i.e.\ $\modulus{j(f)} = j(\modulus{f})$ for all $f \in L^q$).
	
	If $T$ is a positive and power-bounded linear operator on $L^p$ and if $T(L^p) \subseteq j(L^q)$, then $\sprEss(T) < 1$.
\end{theorem}

In Theorem~\ref{thm:uniform-lower-bound-for-spectral-gap-positive} we will give a uniform version of this result. The proof of Theorem~\ref{thm:main-result-positive} relies on the ultrapower approach from Subsection~\ref{subsection:eigenspaces-on-ultra-powers} and on the following lemma:

\begin{lemma} \label{lem:fixed-space-of-positive-operators}
	Under the assumptions of Theorem~\ref{thm:main-result-positive} the fixed space $\fix T$ is finite dimensional.
\end{lemma}
\begin{proof}
	(1) We first show that the assertion is true if $j$ is injective and if $\fix T$ is a \emph{sublattice} of $L^p$, meaning that $\modulus{f} \in \fix T$ for each $f \in \fix T$. 
	
	To this end, first note that $j^{-1}T: L^p \to L^q$ is continuous by the closed graph theorem. In particular, the operator $S := j^{-1}Tj: L^q \to L^q$ is continuous. Moreover, a vector $f \in L^q$ is in the fixed space of $S$ if and only if $j(f)$ is in the fixed space of $T$, i.e.\ $\fix S = j^{-1}(\fix T)$.
	
	The spaces $\fix T$ and $\fix S$ are closed in $L^p$ and $L^q$, respectively. As $\fix T$ is a sublattice of $L^p$, it follows that $\fix S$ is a sublattice of $L^q$ since $j$ is a lattice homomorphism. Now it follows from the representation theorem for abstract $L^p$-spaces \cite[Theorem~2.7.1]{Meyer-Nieberg1991} that $\fix T$, with the norm induced by $L^p$, is itself isometrically lattice isomorphic to an $L^p$-space over some measure space, and likewise it follows that $\fix S$, with the norm induced by $L^q$, is isometrically lattice isomorphic to an $L^q$-space over some measure space.
	
	Yet, the mapping $j|_{\fix S}: \fix S \to \fix T$ is bijective and a lattice homomorphism, hence a lattice isomorphism. As $p \not= q$, Proposition~\ref{prop:L_p-and-L_q-are-non-lattice-isomorphic} in the Appendix shows that this can only be true of $\fix T$ is finite dimensional. \smallskip 
	
	(2) Now we drop the additional assumptions that $j$ be injective and that the fixed space of $T$ be a sublattice of $L^p$; we reduce this general situation to the simpler situation in step~(1).
	
	Since $T$ is power bounded and $L^p$ is reflexive, it follows that $T$ is \emph{mean ergodic}, meaning that the Ces{\`a}ro means $\frac{1}{n}\sum_{k=0}^{n-1} T^k$ converge strongly to a projection $P$ on $L^p$ as $n \to \infty$. The range of $P$ coincides with the fixed space of $T$ and clearly, $P$ is positive. We now define the set $I$ to be the \emph{absolute kernel} of $P$, i.e.\ $I := \{f \in L^p: \, P\modulus{f} = 0\}$. Then $I$ is a closed ideal in the Banach lattice $L^p$ and it is invariant under $T$ and $P$. The factor space $L^p/I$ is again a Banach lattice, and in fact isometrically lattice isomorphic to an $L^p$-space over another measure space. By $q: L^p \to L^p/I$ we denote the quotient mapping; then $q$ is a lattice homomorphism.
	
	Let $T_/,P_/: L^p/I \to L^p/I$ denote the operators induced on the factor space by $T$ and $P$, respectively. We claim that $q|_{\fix T}$ is a bijection from $\fix T$ to $\fix T_/$. Clearly, $q$ maps fixed points of $T$ to fixed points of $T_/$. Moreover, if $f \in \fix T$ and $q(f) = 0$, then $\modulus{f} = \modulus{Pf} \le P\modulus{f} = 0$, so $f = 0$; this shows that $q|_{\fix T}$ is injective. To show surjectivity, note that $T_/$ is mean ergodic with mean ergodic projection $P_/$, so $\fix T_/$ is the range of $P_/$. If $q(f) \in \fix T_/$ (for some $f \in L^p$), we thus have $q(f) = P_/q(f) = q(Pf)$, so $q(f)$ is the image of the vector $Pf \in \fix T$ under $q$.
	
	Since $\fix T$ and $\fix T_/$ are isomorphic vector spaces, it suffices to prove that the fixed space of $T_/$ is finite dimensional. Now, observe that $\fix T_/$ is a sublattice of $L^p/I$: as shown above, every vector in $\fix T_/$ can be written as $q(f)$ for some $f \in \fix T$; we conclude
	\begin{align*}
		T_/ \modulus{q(f)} - \modulus{q(f)} = T_/ q(\modulus{f}) - q(\modulus{f}) = q\left(T\modulus{f} - \modulus{f}\right) = 0
	\end{align*}
	since $P\modulus{T\modulus{f} - \modulus{f}} = TP\modulus{f} - P\modulus{f} = 0$ (where we used that $T\modulus{f} \ge \modulus{f}$ as $T$ is positive and $f \in \fix T$).
	
	Finally, consider the lattice homomorphism $q \circ j: L^q \to L^p/I$, whose range contains the range of $T_/$. We note that $\ker(q \circ j)$ is a closed ideal in $L^q$. Factoring out this kernel in $L^q$ we obtain an induced injective lattice homomorphism $\tilde j: L^q/\ker(q \circ j) \to L^p/I$ with the same range as $q \circ j$. Since $L^q/\ker(q \circ j)$ is also (isometrically lattice isomorphic to) an $L^q$-space, we are now in the situation of step~(1) and can thus conclude that $\fix T_/$ is finite dimensional.
\end{proof}

\begin{remarks_no_number}
	(a) If $T$ is contractive in the above lemma, then we do not need to factor out the ideal $I$ in $L^p$ since it is easily seen that the fixed space of every positive contraction on $L^p$ is a sublattice.
	
	(b) The reduction argument in step~(2) of the proof of Lemma~\ref{lem:fixed-space-of-positive-operators} is not uncommon in the theory of positive operators. For instance, a similar argument is used at the end of the proof of \cite[Theorem~2]{Lin1978}.
\end{remarks_no_number}

For the case of self-adjoint operators on $L^2$, a version of Lemma~\ref{lem:fixed-space-of-positive-operators} (under the assumption that $(\Omega_1,\mu_1) = (\Omega_2, \mu_2)$ is a finite measure space) has been proved by Gross as a part of Theorem~1 in \cite{Gross1972}; in this reference, an explicit bound of the dimension of $\fix T$ in terms of the operator norm of $T$ as an operator from $L^2$ to $L^q$ is given.

\begin{proof}[Proof of Theorem~\ref{thm:main-result-positive}]
	If the spectral radius of $T$ is strictly smaller than $1$, there is nothing to prove, so we assume throughout the proof that $T$ has spectral radius $1$. In particular, $1$ is a spectral value of $T$ since $T$ is positive \cite[Proposition~V.4.1]{Schaefer1974}. 
	
	Note that the kernel of $j$ is a closed ideal in $L^q$ since $j$ is a lattice homomorphism; thus $L^q/\ker j$ is (isometrically lattice isomorphic to) another $L^q$-space, and $j$ induces an injective lattice homomorphism $L^q/\ker j \to L^p$ with the same range as $j$. Therefore, we may assume throughout the rest of the proof that $j$ is injective.
	
	Fix a free ultrafilter $\calU$ on $\bbN$. We first show that the fixed space of the operator $T^\calU$ on $(L^p)^\calU$ is finite dimensional. To this end, we are going to employ Lemma~\ref{lem:fixed-space-of-positive-operators}, but to the operator $T^\calU$ on the $L^p$-space $(L^p)^\calU$ (rather than to the operator $T$ on $L^p$). Note that the space $(L^p)^\calU$ is itself (isometrically lattice isomorphic to) an $L^p$-space over some measure space, and that $T^\calU$ is power-bounded as $T$ is so; hence we have to find an appropriate $L^q$-space and an appropriate lattice homomorphism in order to make the lemma work.
	
	The space $(L^q)^\calU$ is (isometrically lattice isomorphic to) an $L^q$-space over some measure space, and the mapping $j^\calU: (L^q)^\calU \to (L^p)^\calU$ is a lattice homomorphism. Moreover, the range of $T^\calU$ is contained in the range of $j^\calU$. Indeed, let $f^\calU \in (L^p)^\calU$ where $f = (f_n)_{n \in \bbN} \in \ell^\infty(L^p)$. It follows from the closed graph theorem that the mapping $j^{-1}T: L^p \to L^q$ is continuous, so the sequence $(j^{-1}Tf_n)_{n \in \bbN}$ is an element of $\ell^\infty(L^q)$. We thus obtain
	\begin{align*}
		T^\calU f^\calU = \big((Tf_n)_{n \in \bbN}\big)^\calU = \big((jj^{-1}Tf_n)_{n \in \bbN}\big)^\calU = j^\calU \, \big((j^{-1}Tf_n)_{n \in \bbN}\big)^\calU,
	\end{align*}
	so $T^\calU f^\calU$ is in the range of $j^\calU$. Therefore, we can apply Lemma~\ref{lem:fixed-space-of-positive-operators} and conclude that the fixed space of $T^\calU$ is finite dimensional.
	
	Now we employ what we proved in Subsection~\ref{subsection:eigenspaces-on-ultra-powers}: Corollary~\ref{cor:eigenspace-on-ultra-power} tells us that the number $1$ is a pole of the resolvent $\Res(\argument,T)$ and its algebraic multiplicity as an eigenvalue of $T$ is finite. Moreover, we note that the order of the pole equals one since the operator $T$ is power-bounded.
	
	Now it is easy to conclude the proof by combining a few known results from Perron--Frobenius theory. First we employ a theorem which goes originally back to Niiro and Sawashima \cite[Theorem~9.2]{Niiro1966} and which can, in the version that we use here, be found in \cite[Theorem~V.5.5]{Schaefer1974}. The theorem says that, as the spectral radius of our operator $T$ is a pole of the resolvent and as the corresponding spectral projection has finite-dimensional range, every spectral value of $T$ of maximal modulus is a pole of the resolvent. Hence, $T$ has only finitely many spectral values on the unit circle, and each such spectral value is a pole of the resolvent. Moreover, it readily follows from the Neumann series representation of the resolvent that the order of each such pole is dominated by the order of the pole $1$; hence, all unimodular spectral values of $T$ are first order poles of $\Res(\argument,T)$, so the algebraic and the geometric multiplicities of all those eigenvalues coincide.
	
	Thus, it only remains to show that the eigenspace of each unimodular eigenvalue of $T$ is finite dimensional. This follows, for instance, from \cite[Theorem~2]{Lin1978} or from \cite[Theorem~5.3]{Glueck2016} (the assumptions of this theorem are satisfied since the Banach lattice $L^p$ has order continuous norm and the operator $T$ is mean ergodic).
\end{proof}

\begin{remark}
	The above proofs of Lemma~\ref{lem:fixed-space-of-positive-operators} and Theorem~\ref{thm:main-result-positive} show that one does not actually need $T$ to be power-bounded. It suffices that $T$ has spectral radius at most $1$ and that $T$ is mean ergodic -- which, on reflexive Banach spaces, is slightly weaker than being power-bounded.
\end{remark}

To illustrate the generality of Theorem~\ref{thm:main-result-positive} we prove the following corollary for $L^p$-spaces over subsets of $\bbR^d$ with possibly infinite measure.

\begin{corollary} \label{cor:main-result-infinite-measure-space}
	Let $\Omega$ be a Borel measurable subset of $\bbR^d$, endowed with the $d$-dimensional Lebesgue measure $\lambda$, and let $1 \le p < \infty$. Let $T$ be a positive and power-bounded linear operator on $L^p(\Omega,\lambda)$ and assume that there exist numbers $\varepsilon_1,\varepsilon_2 > 0$ such that
	\begin{align*}
		& \int_{\Omega} \modulus{Tf(x)}^{p} (1 + \modulus{x})^{\varepsilon_2} \dx \lambda(x) < \infty \\
		\text{and} \qquad
		& \int_{\Omega} \modulus{Tf(x)}^{p + \varepsilon_1} (1 + \modulus{x})^{\varepsilon_2} \dx \lambda(x) < \infty
	\end{align*}
	for each $f \in L^p(\Omega,\lambda)$. Then $\sprEss(T) < 1$.
\end{corollary}
\begin{proof}
	Let us define $\delta: \Omega \to \bbR$ by $\delta(x) = 1 + \modulus{x}$ for all $x \in \Omega$. Choose $q \in (p,\infty)$ sufficiently close to $p$ such that $q \le p + \varepsilon_1$ and $d(\frac{q}{p} - 1) < \frac{\varepsilon_2}{2}$. From $p < q \le p + \varepsilon_1$ it follows by interpolation that
	\begin{align*}
		\int_{\Omega} \modulus{Tf}^{q} \, \delta^{\varepsilon_2} \dx \lambda < \infty
	\end{align*}
	for each $f \in L^p(\Omega,\lambda)$. Now we choose a real number $\alpha$ which is strictly larger than $d(1 - \frac{p}{q})$ but strictly smaller than $d(1 - \frac{p}{q}) + \frac{\varepsilon_2}{2}\frac{p}{q}$. Then we have
	\begin{align*}
		d(1 - \frac{p}{q}) < \alpha \qquad \text{and} \qquad \frac{q}{p} \alpha < \varepsilon_2,
	\end{align*}
	hence, $\int_{\Omega} \modulus{Tf}^{q} \, \delta^{\frac{q}{p} \alpha} \dx \lambda < \infty$ for each $f \in L^p(\Omega,\lambda)$.
	
	Set $r := q/p \in (1,\infty)$ and choose $r' \in (1,\infty)$ such that $1/r + 1/r' = 1$ (i.e.\ $r' = \frac{q}{q-p}$); moreover, let $\mu$ denote the measure on the Borel $\sigma$-algebra over $\Omega$ which has density $\delta^{\alpha r}$ with respect to the Lebesgue measure, i.e.\ $\dx\mu = \delta^{\alpha r} \dx\lambda$. 
	
	For every function $g \in L^q(\Omega,\mu)$ it follows from H\"{o}lder's inequality that
	\begin{align*}
		\int_\Omega \modulus{g}^p \dx \lambda = \int_\Omega \modulus{g}^p \, \delta^\alpha \, \frac{1}{\delta^\alpha} \dx\lambda \le \left(\int_\Omega \modulus{g}^{pr} \delta^{\alpha r} \dx\lambda\right)^{1/r} \; \left(\int_\Omega \delta^{-\alpha r'}\dx\lambda\right)^{1/r'}, 
	\end{align*}
	and hence
	\begin{align*}
		\norm{g}_{L^p(\Omega,\lambda)} \le \norm{g}_{L^q(\Omega,\mu)} \; \left( \int_\Omega \delta^{-\alpha r'}\dx\lambda \right)^{\frac{1}{r'p}}
	\end{align*}
	We chose $\alpha$ to be strictly larger than $d\frac{q-p}{q}$, hence we have $\alpha r' > d$. This implies that $\int_{\Omega} \delta^{-\alpha r'} \dx \lambda < \infty$, so we have checked that $L^q(\Omega,\mu)$ continuously embeds into $L^p(\Omega,\lambda)$ (and obviously, the embedding is a lattice homomorphism). Moreover, we have already noted above that $\int_\Omega \modulus{Tf}^q \, \delta^{\alpha r}\dx\lambda < \infty$ for all $f \in L^p(\Omega,\lambda)$, so the range of $T$ is contained in $L^q(\Omega,\mu)$. Thus, the assertion follows from Theorem~\ref{thm:main-result-positive}.
\end{proof}

Assume for a moment that the Borel measurable set $\Omega \subseteq \bbR^d$ is bounded; then Corollary~\ref{cor:main-result-infinite-measure-space} becomes a special case of Theorem~\ref{thm:main-result-introduction}. Conversely, it is not difficult to deduce Theorem~\ref{thm:main-result-introduction} from Corollary~\ref{cor:main-result-infinite-measure-space} in case that the space $\Omega$ in the theorem is a bounded and Borel measurable subset of $\bbR^d$ and the measure $\mu$ in the theorem is the Lebesgue measure.

\subsection{Contractive operators}

Wang showed in \cite[Theorem~1.4]{Wang2004} that the positivity assumption in Theorem~\ref{thm:main-result-introduction} (and thus in Theorem~\ref{thm:main-result-positive}) cannot be dropped. However, our next result proves that we can replace it with the assumption that $T$, as well as the operator induced by $T$ on $L^q$, is contractive.

\begin{theorem} \label{thm:main-result-contractive}
	Let $p,q \in (1,\infty)$ be two distinct numbers and let $(\Omega_1,\mu_1)$ and $(\Omega_2,\mu_2)$ be two arbitrary measure spaces. Moreover, let $j: L^q := L^q(\Omega_2,\mu_2) \to L^p := L^p(\Omega_1,\mu_1)$ be an injective lattice homomorphism.
	
	Suppose that $T$ is linear operator on $L^p$ of norm at most $1$, that $TL^p \subseteq j(L^q)$ and that the induced operator $j^{-1}Tj: L^q \to L^q$ also has norm at most $1$. Then $\sprEss(T) < 1$.
\end{theorem}
\begin{proof}
	The main ideas are similar to those used in the proof of Theorem~\ref{thm:main-result-positive}, so we are a bit briefer here. First assume that the number $1$ is a spectral value of $T$. We show that $1$ is a pole of the resolvent $\Res(\argument,T)$ and that its algebraic multiplicity as an eigenvalue of $T$ is finite.
	
	To this end, let us first show that $\fix T$ is finite dimensional. Indeed, the mapping $j|_{\fix (j^{-1}Tj)}$ is a continuous linear bijection between $\fix T$ and $\fix(j^{-1}Tj)$. We note that $T$ is mean ergodic since it is power-bounded and since $L^p$ is reflexive. The associated mean ergodic projection $P: L^p \to L^p$ has $\fix T$ as its range and is contractive since $T$ is so. Thus, as shown by Tzafriri in \cite[Theorem~6]{Tzafriri1969} (based on earlier work of Ando \cite[Theorem~4]{Ando1966}), the range $\fix T$ of $P$ is isometrically isomorphic to an $L^p$-space over some measure space. By the same reasoning we obtain that $\fix(j^{-1}Tj)$ is isometrically isomorphic to an $L^q$-space over some measure space. But $j$ is a Banach space isomorphism between those spaces; according to Theorem~\ref{thm:L_p-and-L_q-are-not-banach-space-isomorphic} in the appendix this can only be true if $\fix T$ is finite dimensional.
	
	Now we proceed as in the proof of Theorem~\ref{thm:main-result-positive} and lift the situation to ultrapowers $(L^p)^\calU$ and $(L^q)^\calU$. Since $j^\calU$ might not be injective in general, we factor out its kernel and obtain an injective lattice homomorphism $J: (L^q)^\calU / \ker(j^\calU) \to (L^p)^\calU$ with the same range as $j^\calU$ (although positivity does not play any further role in this proof, here we need that $j$, and thus $j^\calU$, is a lattice homomorphism to ensure that the quotient space $(L^q)^\calU/\ker(j^\calU)$ is still an $L^q$-space). The operator $T^\calU: (L^p)^\calU \to (L^p)^\calU$ satisfies the same assumptions as $T$. The only property for which this is not obvious is that $J^{-1}T^\calU J$ is contractive on $(L^q)^\calU/\ker(j^\calU)$, but this can be seen as follows: the contractive operator $(j^{-1}Tj)^\calU = (j^{-1}T)^\calU j^\calU$ leaves $\ker(j^\calU)$ invariant (in fact, it even vanishes on $\ker(j^\calU)$), so it induces an operator from $(L^q)^\calU/\ker(j^\calU)$ to $(L^q)^\calU/\ker(j^\calU)$ (which is again contractive), and it is easily checked that this induced operator coincides with $J^{-1}T^\calU J$.
	
	 We can now apply our results from the first part of the proof in order to conclude that $\fix(T^\calU)$ is finite dimensional. Thus, Corollary~\ref{cor:eigenspace-on-ultra-power} shows that $1$ is a pole of the resolvent $\Res(\argument,T)$ and an eigenvalue of $T$ of finite algebraic multiplicity.
	
	Finally, let $\lambda \in \bbC$ be any complex number of modulus $1$. If $\lambda$ is a spectral value of $T$ we can apply what we have just shown to the operator $\overline{\lambda}T$ and thus conclude that $\lambda$ is a pole of the resolvent $\Res(\argument,T)$ and an eigenvalue of $T$ of finite algebraic multiplicity. Hence, $\sprEss(T) < 1$.
\end{proof}

A nice consequence of Theorem~\ref{thm:main-result-contractive} is that the following corollary is true without any additional assumptions on the operator $T$. Let $(\Omega,\mu)$ be a probability space and let $p \in (1,\infty)$. A bounded linear operator $T$ on $L^p(\Omega,\mu)$ is called \emph{hypercontractive} if its range is contained in $L^q(\Omega,\mu)$ for some $q \in (p,\infty)$ and if the norm of $T$ as an operator from $L^p(\Omega,\mu)$ to $L^q(\Omega,\mu)$ does not exceed $1$.

\begin{corollary} \label{cor:hypercontractive-operators}
	Let $(\Omega,\mu)$ be a probability space and let $p \in (1,\infty)$. If $T$ is a hypercontractive linear operator on $L^p(\Omega,\mu)$, then $\sprEss(T) < 1$.
\end{corollary}
\begin{proof}
	Let $q \in (p,\infty)$ be as in the definition of hypercontractivity. Since $(\Omega,\mu)$ is a probability space, the embedding of $L^q := L^q(\Omega,\mu)$ into $L^p := L^p(\Omega,\mu)$ is contractive. Hence, both operators $T:L^p \to L^p$ and $T|_{L^q}: L^q \to L^q$ are contractive, so the assertion follows from Theorem~\ref{thm:main-result-contractive}.
\end{proof}

We remark that the assumptions of Theorem~\ref{thm:main-result-contractive} can be slightly relaxed if $p = 2$ or $q = 2$.

\begin{remark} \label{rem:main-result-hilbert}
	Suppose that, in the situation of Theorem~\ref{thm:main-result-contractive}, the assumption that $T$ and $j^{-1}Tj$ both be contractive is replaced with one of the following two assumptions:
	\begin{enumerate}[\upshape (a)]
		\item We have $p = 2$, the operator $T$ on $L^p = L^2$ has spectral radius $1$ and the induced operator $j^{-1}Tj$ is contractive on $L^q$.
		\item The operator $T$ is contractive on $L^p$ and we have $q = 2$.
	\end{enumerate}
	Then the conclusion $\sprEss(T) < 1$ remains true.
\end{remark}
\begin{proof}
	Note that the contractivity of $T$ and $j^{-1}Tj$ was only used in the proof of Theorem~\ref{thm:main-result-contractive} to show that the fixed spaces of those operators (which are the ranges of their mean ergodic projections) are isomorphic to an $L^p$- and an $L^q$-space, respectively. Now, if one of those operators is not contractive but defined on a Hilbert space, then its fixed space is obviously also a Hilbert space and thus isomorphic to an $L^2$-space over some measure space. The rest of the proof is the same as for Theorem~\ref{thm:main-result-contractive} (note that we need to explicitly impose that assumption $r(T) = 1$ in~(a) since this is not automatic there).
\end{proof}

\section{Uniform estimates} \label{section:uniform-estimates}

As observed in \cite[Proposition~11]{Miclo2015} the essential spectral radius $\sprEss(T)$ in Theorem~\ref{thm:main-result-introduction} can be arbitrarily close to $1$. However, things are different if the norm of $T$ as an operator from $L^p$ to $L^q$ is bounded by a fixed constant. This is the content of our next theorem which complements \cite[Theorem~14]{Miclo2015}; more precisely, \cite[Theorem~14]{Miclo2015} deals exclusively with self-adjoint Markov operators, but it yields explicit bounds for the largest non-unimodular eigenvalues of those operators. By contrast, the subsequent theorem does not assume self-adjointness and replaces $L^2$ with $L^p$, but only yields the \emph{existence} of a bound for $\sprEss(T)$, and it does not give any information about the precise magnitude of eigenvalues which are larger than $\sprEss(T)$.

\begin{theorem} \label{thm:uniform-lower-bound-for-spectral-gap-positive}
	Fix three real numbers $k_1,k_2,k_3 \ge 0$. Assume that, in the situation of Theorem~\ref{thm:main-result-positive}, $j$ is injective, and that we have $\norm{j}_{L^q \to L^p} \le k_1$ and $\norm{j^{-1}T}_{L^p \to L^q} \le k_2$ as well as $\sup_{k \in \bbN_0} \norm{T^k} \le k_3$.
	
	Then it follows that $\sprEss(T) \le c$, where $c = c(p,q, k_1, k_2, k_3) < 1$ is a constant which depends only on $p$, $q$, $k_1$, $k_2$ and $k_3$.
\end{theorem}

Note that the constant $c$ does not explicitly depend on the measure spaces $(\Omega_1,\mu_1)$ and $(\Omega_2,\mu_2)$; however, in the standard situation where $(\Omega_1,\mu_1) = (\Omega_2,\mu_2)$ is a finite measure space, where $q > p$ and where $j$ is the canonical embedding, the operator norm of $j$ does of course depend on the number $\mu_1(\Omega_1)$. Thus, some dependence of $c$ on the underlying measure spaces is implicitly given by the dependence of $c$ on the operator norm of $j$.

Recall that our proof of Theorem~\ref{thm:main-result-positive} relied (via the proof of Theorem~\ref{thm:kernel-on-ultra-power}) on the employment of an ultrapower of a given Banach space; we are now going to use the same technique again, but in addition we also need the slightly more general concept of an \emph{ultraproduct} of several different Banach spaces. To describe elements of ultraproducts we use a similar notation as introduced for the elements of ultrapowers in Subsection~\ref{subsection:ultra-powers}. For details about the theory of ultraproducts we refer to the survey article \cite{Heinrich1980}.

\begin{proof}[Proof of Theorem~\ref{thm:uniform-lower-bound-for-spectral-gap-positive}]
	Assume that the theorem fails. Then we can find (for $n \in \bbN$) measure spaces $(\Omega_1^{(n)},\mu_1^{(n)})$ and $(\Omega_2^{(n)}, \mu_2^{(n)})$, injective lattice homomorphisms $j_n: E_n := L^p(\Omega_1^{(n)}, \mu_1^{(n)}) \to F_n := L^q(\Omega_2^{(n)},\mu_2^{(n)})$ and power-bounded positive linear operators $T_n$ on $E_n$ such that $\sprEss(T_n) \uparrow 1$ and such that $\norm{j_n}_{F_n \to E_n} \le k_1$ and $\norm{j_n^{-1}T_n}_{E_n \to F_n} \le k_2$ as well as $\sup_{k \in \bbN_0} \norm{T_n^k} \le k_3$ for each $n$.
	
	Each $T_n$ has a number $\lambda_n$ of modulus $\modulus{\lambda_n} = \sprEss(T_n) < 1$ in its essential spectrum, and after choosing an appropriate subsequence we may assume that $(\lambda_n)$ converges to a complex number $\lambda$ of modulus $1$. 
	
	Now choose a free ultrafilter $\calU$ on $\bbN$. We define $\tilde E_n := (E_n)^\calU$, $\tilde T_n = (T_n)^\calU$ and $\tilde F_n := (F_n)^\calU / \ker(j_n)^\calU$ for each $n \in \bbN$ and we let $\tilde j_n: \tilde F_n \to \tilde E_n$ denote the injective lattice homomorphism induced by $(j_n)^\calU : (F_n)^\calU \to \tilde E_n$ for each $n \in \bbN$.
	
	For the operators $\tilde j_n$ and $\tilde T_n$ we also have $\norm{\tilde j_n}_{\tilde F_n \to \tilde E_n} \le k_1$ and $\norm{\tilde j_n^{-1}\tilde T_n}_{\tilde E_n \to \tilde F_n} \le k_2$ as well as $\sup_{k \in \bbN_0} \norm{\tilde T_n^k} \le k_3$ for each $n \in \bbN$; the second of those inequalities follows from the fact that, for each $n$, the operator $\tilde j_n^{-1}\tilde T_n: \tilde E_n \to \tilde F_n$ is the composition of the operator $(j_n^{-1}T_n)^\calU: \tilde E_n \to (F_n)^\calU$ with the quotient map $(F_n)^\calU \to \tilde F_n$.
	
	Since $\modulus{\lambda_n} = \sprEss(T_n)$ for each $n$ and since $T_n$ has at most countably many spectral values of modulus strictly larger than $\sprEss(T_n)$, it follows that $\lambda_n$ is contained in the boundary of $\spec(T_n)$. Thus, it follows from Corollary~\ref{cor:eigenspace-on-ultra-power} that, for each index $n$, the number $\lambda_n$ is an eigenvalue of $\tilde T_n$ with infinite dimensional eigenspace. Hence, for each $n \in \bbN$ we can find a sequence of eigenvectors $(f_n^{(k)})_{k \in \bbN} \subseteq \ker(\lambda_n - \tilde T_n)$ of norm $\norm{f_n^{(k)}} = 1$ such that $\norm{f_n^{(k)} - f_n^{(j)}} \ge 1/2$ whenever $j \not= k$.
	
	Now we employ an ultraproduct argument. Let $\tilde E$ and $\tilde F$ as well as $\tilde j: \tilde F \to \tilde E$ and $\tilde T: \tilde E \to \tilde E$ denote the ultraproducts of the spaces $\tilde E_n$ and $\tilde F_n$ and of the operators $\tilde j_n$ and $\tilde T_n$, respectively, along the free ultrafilter $\calU$ (it is not important here that $\calU$ is the same ultrafilter which we used in the first part of the proof). Note that the ultraproduct of the operators $\tilde j_n$ can only by constructed since the embeddings $\tilde j_n$ are norm bounded by $k_1$. The powers of the operator $\tilde T$ are norm bounded by $k_3$.
	
	The fact that $\norm{\tilde j_n^{-1} \tilde T_n}_{\tilde E_n \to \tilde F_n} \le k_2$ for all $n$ implies that the range of $\tilde T$ is contained in the range of $\tilde j$. As $\tilde E$ is an $L^p$-space and $\tilde F$ is an $L^q$-space, we can apply Theorem~\ref{thm:main-result-positive} to conclude that $\sprEss(\tilde T) < 1$.
	
	On the other hand, $\lambda$ is an eigenvalue of $\tilde T$ and for each $k \in \bbN$ the vector $f^{(k)} := \big((f_n^{(k)})_{n \in \bbN}\big)^\calU \in \tilde E$ is a corresponding eigenvector. For $j\not=k$ we have $\norm{f^{(k)} - f^{(j)}} \ge 1/2$, so $\ker(\lambda - \tilde T)$ is infinite dimensional, which contradicts $\sprEss(\tilde T) < 1$.
\end{proof}

We refrain from also discussing a uniform version of Theorem~\ref{thm:main-result-contractive}.

\section{Convergence of positive semigroups} \label{section:applications}

An important task in the study of linear dynamical systems is to give sufficients conditions for the convergence of \emph{operator semigroups}. Let $E$ be a Banach space. An \emph{operator semigroup} (more precisely, a \emph{one-parameter semigroup of linear operators}) on $E$ is a family of bounded linear operators $(T_t)_{t \in (0,\infty)} \subseteq \calL(E)$ such that $T_t T_s = T_{t+s}$ for all $s,t \in (0,\infty)$ (we do not require any continuity assumption with respect to the time parameter $t$ here). The semigroup $(T_t)_{t \in (0,\infty)}$ is called \emph{bounded} if $\sup_{t \in (0,\infty)} \norm{T_t} < \infty$. If $E$ is a Banach lattice and each $T_t$ is a positive operator, then $(T_t)_{t \in (0,\infty)}$ is called a \emph{positive operator semigroup}.

Typical results about the long term behaviour of operator semigroups impose, for instance, regularity and/or spectral assumptions on a semigroup and conclude that $T_t$ converges as $t \to \infty$ either with respect to the strong operator topology or even with respect to the operator norm, depending on the particular assumptions. For instance, it was proved by Lotz \cite[Theorem~4]{Lotz1986} that if $(T_t)_{t \in (0,\infty)}$ is a positive and bounded operator semigroup on a Banach lattice $E$ and if $\sprEss(T_{t_0}) < 1$ for at least one time $t_0$, then $T_t$ converges with respect to the operator norm as $t \to \infty$. If we combine this with Theorem~\ref{thm:main-result-introduction} we immediately obtain the following convergence result.

\begin{theorem} \label{thm:convergence-of-positive-semigroups}
	Let $1 < p < q < \infty$ and let $(\Omega,\mu)$ be a finite measure space. Let $(T_t)_{t \in (0,\infty)}$ be a positive and bounded operator semigroup on $L^p(\Omega,\mu)$. If there exists a time $t_0 \in (0,\infty)$ for which we have $T_{t_0}L^p(\Omega,\mu) \subseteq L^q(\Omega,\mu)$, then $T_t$ converges with respect to the operator norm as $t \to \infty$.
\end{theorem}

Obviously we obtain an even more general result if we make use of Theorem~\ref{thm:main-result-positive} rather than Theorem~\ref{thm:main-result-introduction}.

Theorem~\ref{thm:convergence-of-positive-semigroups} complements several recent results about the long term behaviour of positive operator semigroups; see for instance the article \cite{Mischler2016} which has its focus on spectral theory and growth fragmentation equations, the paper \cite{GerlachLB} on so-called \emph{lower bound methods} for semigroups on $L^1$, the articles \cite{Pichor2016, Gerlach2017, Pichor2018, Pichor2018a, GerlachConvPOS} which all deal with semigroups that dominate integral operators, and the work \cite{Mokhtar-Kharroubi2016} which considers perturbed semigroups on $L^1$-spaces and which is related to the aforementioned series of articles. A good overview of many classical convergence results for positive semigroups can be found in \cite{Arendt1986, Emelyanov2007}; also see the recent monograph \cite{Batkai2017}. The reader who is also interested in the asymptotic theory of general operator semigroups (without positivity assumptions) is referred to \cite[Chapter~V]{Engel2000}, \cite[Chapter~5]{Arendt2011} and to the monograph \cite{Neerven1996}.

\section{Concluding remarks} \label{section:concluding-remarks}

\subsection{End points of the $L^p$-scale} \label{subsection:end-points-of-the-l_p-scale}

In most of our results we only considered the case $p,q \in (1,\infty)$. It is clear that, for instance, Theorem~\ref{thm:main-result-introduction} remains true if $q = \infty$ since $L^\infty(\Omega,\mu)$ embeds into $L^q(\Omega,\mu)$ for any $q \in (1,\infty)$ and thus, in particular, for any $q \in (p,\infty)$. However, much more is acutally true at the end of the $L^p$-scale and it is worthwhile to discuss this in some detail.

We need the following observations from Dunford--Pettis theory: let $T$ be a bounded linear operator between two Banach spaces $E$ and $F$. Recall that $T$ is called \emph{weakly compact} if it maps the closed unit ball in $E$ to a relatively weakly compact subset of $F$. Moreover, $T$ is said to be a \emph{Dunford--Pettis operator} or to be \emph{completely continuous} if, for every sequence $(x_n)$ in $E$ which converges weakly to a vector $x \in E$, the sequence $(Tx_n)$ in $F$ converges in norm to the vector $Tx$; equivalently, $T$ maps weakly compact subsets of $E$ to norm-compact (equivalently: relatively norm-compact) subsets of $F$. Every compact operator between $E$ and $F$ is a Dunford--Pettis operator, and the converse is true if $E$ is reflexive; thus, Dunford--Pettis operators are particularly interesting on non-reflexive Banach spaces. We point out that if $E,F,G$ are Banach spaces, $T: E \to F$ is weakly compact and $S: F \to G$ is a Dunford--Pettis operator, then $ST: E \to G$ is compact. We will make repeated use of this simple observation in the proof of the subsequent theorem.

Now, let $(\Omega,\mu)$ be an arbitrary measure space. Then the spaces $E = L^1(\Omega,\mu)$ and $E = L^\infty(\Omega,\mu)$ are so-called \emph{Dunford--Pettis spaces}, i.e.\ every weakly compact linear operator from $E$ to any Banach space $F$ is a Dunford--Pettis operator; see \cite[Proposition~3.7.9]{Meyer-Nieberg1991}.

Although the subsequent result follows from rather standard arguments from Dunford--Pettis theory we include the proof for the convenience of the reader.

\begin{theorem} \label{thm:end-points-of-the-L_p-scale}
	Let $p,q \in [1,\infty]$ and let $(\Omega_1,\mu_1)$ and $(\Omega_2,\mu_2)$ be two arbitrary measure spaces. Moreover, let $j: L^q := L^q(\Omega_2,\mu_2) \to L^p := L^p(\Omega_1,\mu_1)$ be an injective lattice homomorphism.
	
	Let $T$ be a bounded linear operator on $L^p$ and assume that $TL^p \subseteq j(L^q)$.
	\begin{enumerate}
		\item[(a)] If $q \in (1,\infty]$ and $p = 1$, then $T^2$ is compact.
		\item[(b)] If $q \in [1,\infty)$ and $p = \infty$, then $T^2$ is compact.
		\item[(c)] If $p \in (1,\infty)$ and $q \in \{1,\infty\}$, then $T$ is compact.
	\end{enumerate}
\end{theorem}

The above theorem is, of instance, related to \cite[Proposition 2.1]{Wu2000}, where uniformly integrable operators are considered.

\begin{proof}[Proof of Theorem~\ref{thm:end-points-of-the-L_p-scale}]
	We first observe that, in any case, $j^{-1}T: L^p \to L^q$ is continuous due to the closed graph theorem.

	(a) Let us first show that $T: L^1 \to L^1$ is weakly compact. Consider the case $q \not= \infty$ first. Then the bounded operator $j^{-1}T$ is weakly compact since $L^q$ is reflexive. Since the embedding $j$ is continuous, we conclude that $T = jj^{-1}T$ is weakly compact. Now consider the case $q = \infty$. Then $j$ maps the unit ball of $L^q = L^\infty$ into a subset of $L^1$ of the form $J + iJ$, where $J$ is an order interval in $L^1$. But order intervals in $L^1$ are weakly compact as $L^1$ has order continuous norm (see \cite[Theorem~2.4.2(i) and~(vi)]{Meyer-Nieberg1991}), so $j$ is weakly compact and hence, so is $T = jj^{-1}T$. 
	
	We have thus proved that $T$ is weakly compact. In particular, $T$ is a Dunford--Pettis operator as $L^1$ is a Dunford--Pettis space and hence, $T^2$ is compact (as a composition of a weakly compact operator with a Dunford--Pettis operator).
	
	(b) The mapping $j^{-1}T: L^\infty \to L^q$ maps the unit ball of $L^\infty$ into a subset of $L^q$ of the form $J + iJ$, where $J$ is an order interval in $L^q$. Since $L^q$ has order continuous norm, we again know that order intervals in $L^q$ are weakly compact \cite[Theorem~2.4.2(i) and~(vi)]{Meyer-Nieberg1991}. Hence, $j^{-1}T$ is weakly compact and thus, so is $T = jj^{-1}T$. As $L^\infty$ is a Dunford--Pettis space, we conclude that $T$ is a Dunford--Pettis operator, so $T^2$ is compact (as a composition of a weakly compact operator with a Dunford--Pettis operators).
	
	(c) As $L^p$ is reflexive, the mappings $j^{-1}T: L^p \to L^q$ and $j: L^q \to L^p$ are weakly compact. Thus, $j$ is a Dunford--Pettis operator as $L^q$ is a Dunford--Pettis space (since $q \in \{1,\infty\}$) and consequently, $T = jj^{-1}T$ is compact. 
\end{proof}

The proof of assertion~(c) in the above theorem is actually a special case of the more general observation that a bounded linear operator on a reflexive space which factors through a Dunford--Pettis space is compact. Some of the arguments from Dunford--Pettis theory used in the above proof can be put in a more general context if one considers so-called \emph{principle ideals} in Banach lattices; this is explained in detail in \cite[Section~2]{Daners2017} and in \cite[Section~2]{DanersUnif}.

\subsection{Pitt's theorem on $\ell^p$}

Obviously, one can also apply Theorem~\ref{thm:main-result-positive} to operators on $\ell^p$ whose range is contained in $\ell^q$ for some $q < p$. In this case, however, the assertion of the theorem is an immediate consequence of a much more general result of Pitt which asserts that, for $1 \le q < p < \infty$, every bounded linear operator from $\ell^p$ to $\ell^q$ is compact. For a thorough discussion of Pitt's theorem and possible generalisations we refer the reader to the Appendix of \cite{Rosenthal1969}.

\subsection*{Acknowledgements} 

I am deeply indebted to Delio Mugnolo for several very helpful discussions and comments; he told me about Miclo's article \cite{Miclo2015}, which was the motivation for writing the present paper, he brought Pitt's theorem to my attention and he suggested to generalise the result in Theorem~\ref{thm:main-result-introduction} to Theorem~\ref{thm:main-result-positive}. I am also sincerely indebted to Manfred Wolff for several very fruitful discussions; in particular, he has considerably contributed to bringing the results of Subsection~\ref{subsection:eigenspaces-on-ultra-powers} into their current form, while the first drafts of this paper contained only a much more restricted version of Corollary~\ref{cor:eigenspace-on-ultra-power}. My thanks also go to Christophe Cuny for a very interesting discussion concerning Lotz' paper \cite{Lotz1986} and for pointing out references \cite{ElMachkouri2019} and \cite{Wu2000} to me, as well as to Markus Haase for several comments which helped me to further improve the manuscript.

Part of the work on this article was done while the author was affiliated with the Institute of Applied Analysis, Ulm University, Germany.

\appendix

\section{Banach space isomorphisms between $L^p$- and $L^q$-spaces} \label{section:banach-space-isomorphism-between-L_p-and-L_q-spaces}

The fact that, for instance, $L^p([0,1])$ and $L^q([0,1])$ are not isomorphic as Banach spaces for $p \not= q$ can be shown by techniques from the geometric theory of Banach spaces. Here we give a version of this result for general $L^p$- and $L^q$-spaces, which is needed in the proof of Theorem~\ref{thm:main-result-contractive}. Recall that two Banach spaces are called \emph{isomorphic} if there exists a bounded bijective linear operator between them.

\begin{theorem} \label{thm:L_p-and-L_q-are-not-banach-space-isomorphic}
	Let $p,q \in (1,\infty)$ be two distinct numbers and let $(\Omega_1,\mu_1)$ and $(\Omega_2,\mu_2)$ be two arbitrary measure spaces. If the Banach spaces $L^p(\Omega_1,\mu_1)$ and $L^q(\Omega_2, \mu_2)$ are isomorphic, then they are finite dimensional.
\end{theorem}

Theorem~\ref{thm:L_p-and-L_q-are-not-banach-space-isomorphic} follows from arguments which rely on the \emph{type} and the \emph{cotype} (more precisely, on the \emph{Rademacher type} and \emph{cotype}) of a Banach space. In the literature, one typically finds a version of Theorem~\ref{thm:L_p-and-L_q-are-not-banach-space-isomorphic} which states that the spaces $L^p([0,1])$ and $L^q([0,1])$ are not isomorphic for $p \not= q$. The slightly more general assertion of Theorem~\ref{thm:L_p-and-L_q-are-not-banach-space-isomorphic} can be proved in exactly the same way; for the convenience of the reader -- and since we could not find a reference for exactly this version of the result in the literature -- we sketch the main steps of the proof.

\begin{proof}[Sketch of the proof of Theorem~\ref{thm:L_p-and-L_q-are-not-banach-space-isomorphic}]
	We use the notions of (Rademacher) type and cotype of a Banach space which can, for instance, be found in \cite[Definition~6.2.10]{Albiac2006}. The point about those notions is that the optimal type and the optimal cotype of a space are invariant under Banach space isomorphisms. 
	
	Now assume that $L^p := L^p(\Omega_1,\mu_1)$ and $L^q := L^q(\Omega_2,\mu_2)$ are isomorphic and infinite dimensional. According to \cite[Theorem~6.2.14]{Albiac2006} the optimal types of those spaces are $\min \{p,2\}$ and $\min\{q,2\}$, respectively, while their optimal cotypes are $\max\{p,2\}$ and $\max\{q,2\}$, respectively. Now we distinguish the following four cases:
	\begin{enumerate}[\upshape (i)]
		\item $p < 2$, $q \le 2$,
		\item $p < 2$, $q > 2$,
		\item $p \ge 2$, $q < 2$,
		\item $p \ge 2$, $q \ge 2$. 
	\end{enumerate}
	For each case we obtain a contradiction (since $p \not= q$): in the cases (i)--(iii) the spaces $L^p$ and $L^q$ have distinct optimal type and in case~(iv) they have distinct optimal cotype.
\end{proof}

\begin{remark_no_number}
	Regarding the proof of Theorem~\ref{thm:L_p-and-L_q-are-not-banach-space-isomorphic} the following remark seems to be in order: in the reference \cite[Theorem~6.2.14]{Albiac2006} the assumption that the $L^p$-space under consideration be infinite dimensional is not stated explicitly. However, this assumption is needed for the optimality assertion for both the type and the cotype. Indeed, every finite dimensional Banach space is isomorphic to a Hilbert space and thus has optimal type and optimal cotpye $2$. In the proof of \cite[Theorem~6.2.14]{Albiac2006} the infinite dimension of the space is needed at the very end, where the fact that $L^p$ contains the sequence space $\ell^p$ is used.
\end{remark_no_number}

\section{Lattice isomorphisms between $L^p$- and $L^q$-spaces} \label{section:lattice-isomorphism-between-L_p-and-L_q-spaces}

Theorem~\ref{thm:L_p-and-L_q-are-not-banach-space-isomorphic} employs non-trivial concepts from geometry of Banach spaces. While we need this result in the proof of Theorem~\ref{thm:main-result-contractive}, our other main theorem (Theorem~\ref{thm:main-result-positive}) only makes use of the much simpler observation that infinite dimensional $L^p$- and $L^q$-spaces are not \emph{lattice}-isomorphic for $p \not= q$. We think it is instructive to include an elementary proof of this result (although it is, of course, a special case of the much deeper Theorem~\ref{thm:L_p-and-L_q-are-not-banach-space-isomorphic}):

\begin{proposition} \label{prop:L_p-and-L_q-are-non-lattice-isomorphic}
	Let $p,q \in [1,\infty)$ be two distinct numbers, let $(\Omega_1,\mu_1)$ and $(\Omega_2,\mu_2)$ be arbitrary measure spaces and assume that $L^p := L^p(\Omega_1,\mu_1)$ and $L^q := L^q(\Omega_2,\mu_2)$ are isomorphic as Banach lattices (i.e.\ there exists a lattice isomorphism $L^q \to L^p$). Then $L^q$ (and hence $L^p$) has finite dimension.
\end{proposition}

For the proof we need the simple observation that, in every infinite dimensional Banach lattice $E$, there exists a sequence $(x_k)_{k \in \bbN} \subseteq E_+$ of normalised and pairwise disjoint vectors (i.e.\ $\norm{x_k} = 1$ for all indices $k$ and $x_j \land x_k = 0$ whenever $j \not= k$).

\begin{proof}[Proof of Proposition~\ref{prop:L_p-and-L_q-are-non-lattice-isomorphic}]
	We may assume that $p < q$. Assume for a contradiction that $L^q$ is infinite-dimensional. Then there exists a sequence $(f_k)_{k \in \bbN}$ of normalised and pairwise disjoint vectors $0 \le f_k \in L^q$. The series $\sum_{k=1}^\infty \frac{f_k}{k^{1/p}}$ converges in $L^q$ since, for $1 \le m \le n$, we have
	\begin{align*}
		\Big\lVert\sum_{k=m}^n \frac{f_k}{k^{1/p}} \Big\rVert_q^q = \sum_{k=m}^n \frac{\norm{f_k}_q^q}{k^{q/p}} = \sum_{k=m}^n \frac{1}{k^{q/p}} \to 0 \quad \text{as } m,n \to \infty.
	\end{align*}
	Now, let $J: L^q \to L^p$ be a lattice isomorphism. Then the vectors $g_k := Jf_k \in L^p$ are also pairwise disjoint. As $J$ is continuous, the series $\sum_{k=1}^\infty \frac{g_k}{k^{1/p}}$ converges in $L^p$. However, the mapping $J^{-1}$ is continuous, too, so we have $\norm{g_k}_p \ge c \norm{f_k}_q = c$ for a constant $c > 0$ and all indices $k$. This implies that, for $1 \le m \le n$,
	\begin{align*}
		\Big\lVert\sum_{k=m}^n \frac{g_k}{k^{1/p}} \Big\rVert_p^p = \sum_{k=m}^n \frac{\norm{g_k}_p^p}{k} \ge \sum_{k=m}^n \frac{c^p}{k} \not\to 0 \quad \text{as } m,n \to \infty.
	\end{align*}
	This is a contradiction.
\end{proof}

\bibliographystyle{plain}
\bibliography{literature}

\end{document}